\documentclass[12pt, reqno]{amsart}
\usepackage{graphicx}
\usepackage{amscd}
\usepackage{amsmath}
\usepackage{amsfonts}
\usepackage{amssymb}
\usepackage{cite}
\usepackage{comment}
\usepackage{enumitem}

\newtheorem{theorem}{Theorem}[section]

\newtheorem{lemma}{Lemma}[section]
\newtheorem{definition}{Definition}[section]
\newtheorem{remark}{Remark}[section]

\newcounter{romnum}

\begin{document}

\title{Lyapunov-Type Inequalities for a Fractional Boundary Value Problem with a Fractional Boundary Condition}

\author{Sougata Dhar$^1$ and Jeffrey T. Neugebauer$^2$}



\address{$^1$ Department of Mathematics, University of Connecticut, Storrs, CT 06269 USA \\ $^2$ Department of Mathematics and Statistics, Eastern Kentucky University, Richmond, KY 40475 USA}

\def\R{\mathbb{R}}
\def\N{\mathbb{N}}
\def\P{\mathcal{P}}
\def\O{\Omega}
\def\e{\epsilon}
\def\p{\prime}
\def\d{\delta}
\def\ds{\mathrm{d}s}
\def\a{\alpha}
\def\B{\beta}
\def\Da{D_{a^+}^\alpha}
\def\Db{D_{a^+}^\beta}
\def\Du{D_{0^+}^\mu}
\def\C01{C^{(j-1)}[0,1]}
\def\pj1 {\dfrac{\partial^{n-2}}{\partial t^{n-2}}}
\def\pjj1 {\dfrac{\partial^{j}}{\partial t^{j}}}
\def\fus{f(u(s), u'(s), \dots , u^{(n-2)}(s))}
\def\fut{f(u(t), u'(t), \dots , u^{(n-2)}(t))}
\def\fusj{f(u(s), u'(s), \dots , u^{(j)}(s))}
\def\M{\textit{\textbf{M}}}
\def\ds{\displaystyle}

\begin{abstract} 
In this paper, we consider a linear fractional differential equation with fractional boundary conditions.
First, by obtaining Green's function, we derive the Lyapunov-type inequalities for
such boundary value problems. Furthermore, we use the contraction mapping
theorem to study the existence of a unique solution to a nonlinear problem.
\end{abstract}
\maketitle

\setcounter{equation}{0}
\section{Introduction}

For the second-order linear differential equation
\begin{equation}
u''+q(t)u=0,\quad t\in(a,b)\label{1.1}
\end{equation}
with $q\in C([a,b],\mathbb{R})$, it is known that if \eqref{1.1} has a nontrivial solution $u$ with $u(a)=u(b)=0$, then
\begin{equation}
\int_{a}^{b}\left|q\left(t\right)\right|dt>\frac{4}{b-a}.\label{1.2}
\end{equation}
This result is known as the Lyapunov inequality, see \cite{AML, B}.

It was first noticed by Wintner \cite{W} and later by several other
authors that inequality (\ref{1.2}) can be improved by replacing
$|q(t)|$ by $q_{+}(t):=\max\{q(t),0\}$, the nonnegative part of
$q(t)$.

The Lyapunov inequality was extended by Hartman \cite[Chapter XI]{H}
to the more general equation
\begin{equation}
(r(t)u')'+q(t)u=0,\label{1.4}
\end{equation}
where $q,r\in C\left(\left[a,b\right],\mathbb{R}\right)$ when it was shown that if (\ref{1.4}) has a nontrivial solution $u$
satisfying $u(a)=u(b)=0$ and $u(t)\neq0$ for $t\in(a,b)$, then
\[
\int_{a}^{b}q_{+}\left(t\right)dt>\frac{4}{\int_{a}^{b}r^{-1}\left(t\right)dt}.
\]

These Lyapunov inequalities have been used as an important tool in oscillation,
disconjugacy, control theory, eigenvalue problems, and many other areas of differential
equations. Due to their importance in applications, they have been extended in
various directions by many authors.
For more on Lyapunov-type inequalities, we refer the reader to
\cite{DK1,DK2,DK3,DK4,DK5,DK6,DK7,DK8,DK9,DK10} and the references cited therein.

Recently, fractional differential equations
have gained a considerable attention for their applications in the mathematical modeling of systems and
processes in the fields of physics, mechanics, chemistry, aerodynamics, nonlinear dynamics, and system theory \cite{BS16, CV18, DNV11, DR18}.
Due to useful applications in the boundary value problems (BVPs),
a subsequent search for the Lyapunov-type inequalities has also
begun in the direction of fractional calculus.
Ferreira first obtained Lyapunov-type inequalities for fractional
differential equations with pointwise boundary conditions (BCs).
In \cite{F2}, he considered the Riemann-Liouville fractional
differential equation
\begin{equation}
\Da u+q(t)u=0,\quad 1<\alpha\le 2,\label{1.5}
\end{equation}
where $q\in C([a,b],\mathbb{R})$, and showed that if \eqref{1.5} has a nontrivial solution $u$
satisfying $u(a)=u(b)=0$, then
\begin{equation}
\int_{a}^{b}|q(t)|dt>\Gamma(\alpha)\Big(\frac{4}{b-a}\Big)^{\alpha-1}.\label{1.6}
\end{equation}
In \cite[Theorem 2.3]{DK5}, Dhar and Kong improved \eqref{1.6} by replacing
$|q(t)|$ by $q_+(t)$. Moreover, they
obtained the Lyapunov-type inequalities for a fractional BVP
consisting of Eq.~\eqref{1.5} and the integral BCs
\begin{equation}
I_{a^{+}}^{2-\alpha}u(a^{+})=I_{a^{+}}^{2-\alpha}u(b)=0,\label{1.6*}
\end{equation}
where $I_{a^{+}}^{2-\alpha}u$ is the Riemann-Liouville fractional integral of
$u(t)$ of order $2-\alpha$.

When $\alpha=2$, the results in \cite{F2} and \cite{DK5}
lead to the classical Lyapunov inequality.
For more Lyapunov-type inequalities involving the Riemann-Liouville and
Caputo fractional derivatives, we refer the reader to
\cite{F1, JS, RB} and the references cited therein.

In this paper, we consider a Riemann-Liouville fractional BVP consisting of the
equation
\begin{equation*}
\Da u + q(t)u=0,
\end{equation*}
together with the boundary conditions (BCs)
\begin{equation*}
u(a)=0,\quad \Db u(b)=0,
\end{equation*}
where $\alpha\in(1,2]$,
$\beta\in[0,\alpha-1]$, $\Da$, $\Db$ are Riemann-Liouville derivatives of order
$\alpha$ and $\beta$, respectively, and $q\in C([a,b],\mathbb{R})$.
We obtain Lyapunov-type inequalities and use them to study the
nonexistence of a nontrivial solution of certain BVPs. Furthermore, by
using the contraction mapping theorem, we also establish a criterion
for the existence of a nontrivial solution for a nonlinear fractional BVP.

This paper is organized as follows. After this introduction, we recall
some basic definitions of fractional calculus in Section 2.
Section 3 contains the main results regarding the Lyapunov-type
inequalities. Finally, in Section 4, we obtain a criterion for the nonexistence of nontrivial solutions of a linear BVP and the existence of a unique
solution of a nonlinear fractional BVP.

\section{Background Materials and Preliminaries}

For the convenience of the reader, here we present the necessary definitions
and lemmas from fractional calculus theory in the sense of Riemann-Liouville.
These results can be found in the books \cite{KST, MR, P, SKM}.

\begin{definition}\label{d1}
Let $\nu>0$. The Riemann-Liouville fractional integral of the function $u:[a,b]\to\R$ of
order $\nu$, denoted $I_{a^+}^{\nu} u$, is defined as
\begin{equation*}
I_{a^+}^{\nu} u(t)=\dfrac{1}{\Gamma (\nu)}\int^t_a (t-s)^{\nu-1}u(s)ds,
\end{equation*}
where $\Gamma(\nu)=\int_{0}^{\infty}t^{\nu-1}e^{-t}dt$ is the
gamma function, provided the right-hand side is pointwise defined on $\mathbb{R}^{+}$.
\end{definition}

\begin{definition}\label{d2}
Let $n$ denote a positive integer and assume $n-1< \alpha \le n$.
The Riemann-Liouville fractional derivative of order $\alpha$ of the function
$u:[a,b]\to\R$, denoted $\Da u$, is defined as
\begin{equation*}
\Da u(t)=\dfrac{1}{\Gamma(n-\alpha)}\dfrac{d^n}{dt^n}\int^t_a (t-s)^{n-\alpha-1}u(s)ds=D^{n}I_{a^+}^{n-\alpha} u(t),
\end{equation*}
provided the right-hand side is pointwise defined on $\mathbb{R}^{+}$.
\end{definition}

In the following, unless otherwise mentioned, we use $ D_{a+}^{\alpha} u(t)$ to denote the
fractional derivative of $u(t)$ with order $\alpha$ and $D^{j}u(t)$ to denote the classical
derivative of order $j$ of $u(t)$ with $j$ being a nonnegative integer.
We recall a few well-known properties of the Riemann-Liouville fractional
derivatives and integrals to construct and analyze the family of Green's functions.
Let $u\in L_{1}[a,b]$. Then
\begin{equation}\label{2.1}
I_{a+}^{\nu_{1}}I_{a+}^{\nu_{2}}u(t)=I_{a+}^{\nu_{1}+\nu_{2}}u(t)=I_{a+}^{\nu_{2}}I_{a+}^{\nu_{1}}u(t), \quad \nu_{1}, \nu_{2} >0;
\end{equation}
\begin{equation}\label{2.2}
D_{a+}^{\nu_{1}}I_{a+}^{\nu_{2}}u(t)=I_{a+}^{\nu_{2}-\nu_{1}}u(t), \quad \mbox{ if } 0\le \nu_{1} \le \nu_{2};
\end{equation}
\begin{equation*}
 D_{a+}^{\alpha}I_{a+}^{\alpha}u(t) =u(t);
\end{equation*} and
\begin{equation}\label{2.3}
I_{a+}^{\alpha}D_{a+}^{\alpha}u(t) = u(t)+\sum_{i=1}^{n}c_{i}(t-a)^{\alpha -n +(i-1)},
\end{equation}
where $c_i\in \R$ for $1\le i\le n$.
The property \eqref{2.1} is referred to as the semigroup property for the fractional integral.

It follows from Definition \ref{d1} and \ref{d2} that
$$
I_{a+}^{\nu_{2}}(t-a)^{\nu_{1}}=\frac{\Gamma (\nu_{1} +1)}{\Gamma (\nu_{2} +\nu_{1} +1)}(t-a)^{\nu_{2} +\nu_{1}},\quad \nu_{1} > -1, \nu_{2} \ge 0,
$$
and
\begin{equation}\label{2.4}
D_{a+}^{\nu_{2}}(t-a)^{\nu_{1}} =\frac{\Gamma (\nu_{1} +1)}{\Gamma (\nu_{1}+1 -\nu_{2})}(t-a)^{\nu_{1}-\nu_{2}}, \quad \nu_{1} > -1, \nu_{2} \ge 0,
\end{equation}
where it is assumed that $\nu_{2} -\nu_{1}$ is not a positive integer. If $\nu_{2} -\nu_{1}$ is a positive integer, then the right-hand side of \eqref{2.4} vanishes. To see this, appeal to the convention that $\frac{1}{\Gamma (\nu_{1}+1 -\nu_{2})}=0$ if $\nu_{2} -\nu_{1}$ is a positive integer.

\section{Main Results}

We now consider the fractional boundary value problem consisting of the differential equation
\begin{equation}\label{1}
\Da u + q(t)u=0,~t\in[a,b],
\end{equation}
together with the boundary conditions
\begin{equation}\label{2}
u(a)=0,\quad \Db u(b)=0,
\end{equation}
where $\alpha\in(1,2]$,
$\beta\in[0,\alpha-1]$, $\Da$, $\Db$ are Riemann-Liouville derivatives of order
$\alpha$ and $\beta$, respectively, and $q\in C([a,b],\mathbb{R})$. First, we present Green's function corresponding to the BVP \eqref{1}, \eqref{2}.

\begin{lemma}\label{l3.1}
Let $h\in C([a,b],\mathbb{R})$, $\alpha\in(1,2]$, and $\beta\in[0,\alpha-1]$.
Then the unique solution of the BVP consisting of the equation
\begin{equation}\label{lde}
\Da u+h(t)=0,\quad t\in[a,b],
\end{equation}
and the BCs \eqref{2} is
\begin{equation*}
u(t)= \int_{a}^{b}G(t,s)h(s)ds,\quad t\in[a,b],
\end{equation*}
where
\begin{equation}\label{gf}
G(t,s)=\frac{1}{\Gamma(\alpha)}\left\{\begin{array}{lr}\frac{(t-a)^{\alpha -1}(b-s)^{\alpha-1-\beta}}{(b-a)^{\alpha-1-\beta}}-(t-s)^{\alpha-1}, & a\le s \le t\le b,\\\\
\frac{(t-a)^{\alpha -1}(b-s)^{\alpha-1-\beta}}{(b-a)^{\alpha-1-\beta}}, & a\le t\le s\le b.\end{array}\right.
\end{equation}
\end{lemma}

\begin{proof}
We use \eqref{2.3} to reduce \eqref{lde} to an equivalent integral equation
$$
u(t)= -I_{a+}^{\alpha}h(t)+c_1(t-a)^{\alpha-2}+c_2(t-a)^{\alpha-1}.
$$
The BC $u(a)=0$ implies $c_1=0$,
and hence
$$
u(t)= -I_{a+}^{\alpha}h(t)+c_{2}(t-a)^{\alpha -1}.
$$
Note that $0\le \beta<\alpha$. Applying $\Db$ on both sides and using
\eqref{2.2} and \eqref{2.4}, we have
\begin{align*}
\Db u(t)&=-\Db \left(I_{a+}^{\alpha} h(t)+c_{2}(t-a)^{\alpha -1}\right)\\
&= -I_{a+}^{\alpha-\beta} h(t)+c_2\frac{\Gamma(\alpha)}{\Gamma(\alpha-\beta)}(t-a)^{\alpha-1-\beta}.
\end{align*}
Since $\Db u(b)=0$, it is easy to see that
$$
c_2=\frac{1}{\Gamma(\alpha)(b-a)^{\alpha-1-\beta}}\int_a^b(b-s)^{\alpha-1-\beta}h(s)ds.
$$
Therefore, the unique solution of problem \eqref{lde}, \eqref{2} is
\begin{align*}
u(t)&=\frac{-1}{\Gamma(\alpha)}\int_a^t(t-s)^{\alpha-1}h(s)ds+\frac{(t-a)^{\alpha-1}}{\Gamma(\alpha)(b-a)^{\alpha-1-\beta}}\int_a^b(b-s)^{\alpha-1}h(s)ds\\
&=\frac{1}{\Gamma(\alpha)}\int_a^t\left\{\frac{(t-a)^{\alpha-1}(b-s)^{\alpha-1}}{(b-a)^{\alpha-1-\beta}}-(t-s)^{\alpha-1}\right\}h(s)ds\\
&\phantom{=}+\frac{1}{\Gamma(\alpha)}\int_t^b\frac{(t-a)^{\alpha-1}(b-s)^{\alpha-1}}{(b-a)^{\alpha-1-\beta}}h(s)ds\\
&=\int_a^bG(t,s)h(s)ds.
\end{align*}
The proof is complete.
\end{proof}

\begin{lemma}\label{green1}
Green's function $G(t,s)$ given in
\eqref{gf} satisfies the following properties.
\begin{enumerate}
\item $G(t,s)\ge0$ for $(t,s)\in[a,b]\times[a,b]$.
\item $\max_{t\in[a,b]}G(t,s)\le G(s,s)$ for $s\in[a,b]$.
\item $G(s,s)$ has a unique maximum at $s^*=\frac{(\alpha-1)b+(\alpha-1-\beta)a}{2\alpha-2-\beta}$ given by
\begin{equation}\label{3.4}
G(s^*,s^*)=\dfrac{1}{\Gamma(\alpha)}\left(\frac{(b-a)(\alpha-1)}{2\alpha-2-\beta}\right)^{\alpha-1}\left(\frac{\alpha-1-\beta}{2\alpha-2-\beta}\right)^{\alpha-1-\beta}.
\end{equation}
\end{enumerate}
\end{lemma}
\begin{proof}
Define
$$
g_1(t,s)=\frac{(t-a)^{\alpha -1}(b-s)^{\alpha-1-\beta}}{(b-a)^{\alpha-1-\beta}}-(t-s)^{\alpha-1},
$$
for $a\le s\le t\le b$ and
$$
g_2(t,s)=\frac{(t-a)^{\alpha -1}(b-s)^{\alpha-1-\beta}}{(b-a)^{\alpha-1-\beta}},
$$
for $a\le t\le s\le b$. First, we point out that
\begin{align*}
\frac{b-s}{b-a}-\frac{t-s}{t-a}&=\frac{(b-s)(t-a)-(t-s)(b-a)}{(b-a)(t-a)}\\
&=\frac{(b-t)(s-a)}{(b-a)(t-a)}\ge0,
\end{align*}
or 
$$
\frac{b-s}{b-a}\ge\frac{t-s}{t-a},
$$ 
for $a\le s\le t\le b$.

Now,
\begin{align*}
g_1(t,s)&=\frac{(t-a)^{\alpha -1}(b-s)^{\alpha-1-\beta}}{(b-a)^{\alpha-1-\beta}}-(t-s)^{\alpha-1}\\
&=(t-a)^{\alpha-1}\left(\frac{b-s}{b-a}\right)^{\alpha-1-\beta}-(t-a)^{\alpha-1}\left(\frac{t-s}{t-a}\right)^{\alpha-1}\\
&=(t-a)^{\alpha-1}\left[\left(\frac{b-s}{b-a}\right)^{\alpha-1-\beta}-\left(\frac{t-s}{t-a}\right)^{\alpha-1}\right].
\end{align*}
Since $0\le\dfrac{b-s}{b-a}\le1$, $0\le\dfrac{t-s}{t-a}\le1$, and $\alpha-1-\beta\le\alpha-1$, one has
$$
\left(\frac{b-s}{b-a}\right)^{\alpha-1-\beta}\ge \left(\frac{b-s}{b-a}\right)^{\alpha-1}\ge \left(\frac{t-s}{t-a}\right)^{\alpha-1}.
$$
So $g_1(t,s)\ge0$ for $a\le s\le t\le b$.

Now
\begin{align*}
\dfrac{\partial }{\partial t}g_1(t,s)&=(\alpha-1)\frac{(t-a)^{\alpha -2}(b-s)^{\alpha-1-\beta}}{(b-a)^{\alpha-1}}-(\alpha-1)(t-s)^{\alpha-2}\\
&=(\alpha-1)(t-a)^{\alpha-2}\left(\frac{b-s}{b-a}\right)^{\alpha-1-\beta}-(\alpha-1)(t-a)^{\alpha-2}\left(\frac{t-s}{t-a}\right)^{\alpha-2}\\
&=(\alpha-1)(t-a)^{\alpha-2}\left[\left(\frac{b-s}{b-a}\right)^{\alpha-1-\beta}-\left(\frac{t-s}{t-a}\right)^{\alpha-2}\right].
\end{align*}
Since $0\le\dfrac{b-s}{b-a}\le1$, $0\le\dfrac{t-s}{t-a}\le1$, $\alpha-1-\beta\ge\alpha-2$, and $\alpha-2\le 0$, we have $$\left(\frac{b-s}{b-a}\right)^{\alpha-1-\beta}\le \left(\frac{b-s}{b-a}\right)^{\alpha-2}\le\left(\frac{t-s}{t-a}\right)^{\alpha-2}.$$
So $\frac{\partial }{\partial t}g_1(t,s)\le0$ for $a\le s\le t\le b$. Thus $g_1(t,s)$ is a decreasing function with respect to $t$, implying $g_1(t,s)\le g_1(s,s)$ for all $t\in[s,b]$.

It is easy to see that $g_2(t,s)\ge 0$. Moreover,
$$
\dfrac{\partial }{\partial t}g_2(t,s)=(\alpha-1)\frac{(t-a)^{\alpha -2}(b-s)^{\alpha-1-\beta}}{(b-a)^{\alpha-1-\beta}}\ge 0,
$$
for $a\le t\le s\le b$. So $g_2(t,s)$ is increasing with respect to $t$ implying $g_2(t,s)\le g_2(s,s)$ for all $t\in[a,s]$. Thus (1) and (2) hold.

To prove (3), we define
\begin{equation}\label{3.5}
g(s):=G(s,s)=\frac{(s-a)^{\alpha -1}(b-s)^{\alpha-1-\beta}}{(b-a)^{\alpha-1-\beta}\Gamma(\alpha)}.
\end{equation}
Then $g(a)=g(b)=0$ and $g(s)>0$ on $(a,b)$. By Rolle's theorem,
there exists $s^{*}\in(a,b)$ such that $g(s^{*})=\max_{s\in[a,b]} g(s)$,
i.e., $g'(s^{*})=0$. Note that
\begin{align*}
g'(s)&=\frac{(\alpha-1)(s-a)^{\alpha-2}(b-s)^{\alpha-1-\beta}-(\alpha-1-\beta)(s-a)^{\alpha-1}(b-s)^{\alpha-2}}{(b-a)^{\alpha-1-\beta}\Gamma(\alpha)}\\
&=\frac{(s-a)^{\alpha-2}(b-s)^{\alpha-2-\beta}\left[(\alpha-1)(b-s)-(\alpha-1-\beta)(s-a)\right]}{(b-a)^{\alpha-1-\beta}\Gamma(\alpha)}.
\end{align*}
Hence $g'(s^*)=0$ when
$$
s^*=\frac{(\alpha-1)b+(\alpha-1-\beta)a}{2\alpha-2-\beta}.
$$
Notice
$$
s^*>\frac{(\alpha-1)a+(\alpha-1-\beta)a}{2\alpha-2-\beta}=a,
$$
and
$$
s^*<\frac{(\alpha-1)b+(\alpha-1-\beta)a}{2\alpha-2-\beta}<\frac{(\alpha-1)b+(\alpha-1-\beta)b}{2\alpha-2-\beta}=b,
$$
so $s^*$ is well-defined. Replacing $s^*$ in \eqref{3.5} we see that \eqref{3.4} holds.
\end{proof}

We remark here that if $\beta\in(\alpha-1,1]$, then properties (1) and (2) from Lemma \ref{green1} still hold. However, the function $g(s)$ defined in the proof has a singularity at $b$ when $\beta>\alpha-1$. Hence $G(s,s)$ does not have a maximum value, which is not surprising, since in this case, $G(s,s)$ is only defined for $s\in[a,b)$.

\begin{lemma}\label{l3.2}
Let $G(t,s)$ be given by \eqref{gf}. Then
\begin{equation}\label{int}
\int_a^bG(t,s)ds\le \frac{(\alpha-1)^{\alpha-1}}{(\alpha-\beta)^\alpha \Gamma(\alpha+1)}
(b-a)^\alpha.
\end{equation}
\end{lemma}

\begin{proof}
When using the expression of $G(t,s)$ in \eqref{gf}, it follows that
\begin{align}\nonumber
\int_{a}^{b}G(t,s)ds&=\frac{1}{\Gamma(\alpha)}\left[\frac{(t-a)^{\alpha-1}}{(b-a)^{\alpha-1-\beta}}\int_{a}^{b}(b-s)^{\alpha-1-\beta}ds-\int_{a}^{t}(t-s)^{\alpha-1}ds\right]\\\nonumber
&=\frac{1}{\Gamma(\alpha)}\left[\frac{(t-a)^{\alpha-1}}{(b-a)^{\alpha-1-\beta}}\frac{(b-a)^{\alpha-\beta}}{\alpha-\beta}-\frac{(t-a)^\alpha}{\alpha}\right]\\
&=\frac{(t-a)^{\alpha-1}}{\Gamma(\alpha+1)}\left[\frac{\alpha}{\alpha-\beta}(b-a)-(t-a)\right]. \label{a1}
\end{align}
We denote
\begin{equation}\label{a2}
f(t):=\frac{(t-a)^{\alpha-1}}{\Gamma(\alpha+1)}\left[\frac{\alpha}{\alpha-\beta}(b-a)-(t-a)\right],\quad t\in[a,b].
\end{equation}
Let $c:=a+\frac{\alpha}{\alpha-\beta}(b-a)$.
Clearly, $f(a)=f(c)=0$, and $f(t)>0$ on $(a,c)$.
Since $\alpha-\beta\le\alpha$, we have $b\le c$ with the equality holding
only when $\beta=0$.
By Rolle's theorem, there exists $t^{*}\in(a,c)$ such that $f(t^{*})=\max_{t\in[a,c]} f(t)$,
i.e., $f'(t^{*})=0$. Note that
\begin{equation}\label{a3}
f'(t)=\frac{(t-a)^{\alpha-2}}{\Gamma(\alpha)}\left[\frac{\alpha-1}{\alpha-\beta}(b-a)-(t-a)\right].
\end{equation}
It is easy to see that $f'(t)=0$ only at
$t=t^{*}=a+\frac{\alpha-1}{\alpha-\beta}(b-a)$.
Again, $\alpha-1\le\alpha-\beta$ implies $t^*\le b$ with the equality holding
only when $\beta=1$.
Hence $f(t)$ has a unique maximum at $t^{*}\in[a,b]\subseteq[a,c]$ given by
\begin{equation*}%\label{a7}
\max_{t\in[a,c]}f(t)=\max_{t\in[a,b]}f(t)=f(t^{*})=\frac{(\alpha-1)^{\alpha-1}}{(\alpha-\beta)^\alpha\Gamma(\alpha+1)}(b-a)^\alpha.
\end{equation*}
The proof is complete.
\end{proof}

Now we present a Lyapunov-type inequality for \eqref{1}, \eqref{2}.
\begin{theorem}\label{main1}
Assume \eqref{1} has a nontrivial solution $u$
satisfying \eqref{2} and $u(t)\neq0$ on $(a,b)$. Then
\begin{equation}\label{lti1}
\int_{a}^{b}q_{+}(t)dt> \Gamma(\alpha)\left(\frac{2\alpha-2-\beta}{(b-a)(\alpha-1)}\right)^{\alpha-1}\left(\frac{2\alpha-2-\beta}{\alpha-1-\beta}\right)^{\alpha-1-\beta}.
\end{equation}
\end{theorem}

\begin{proof}
Let $u$ be a solution of \eqref{1}, \eqref{2}. Then $u$ satisfies
$$u(t)=\int^b_a G(t,s)q(s)u(s)ds.$$
Without loss of generality, assume $u(t)>0$ on $(a,b)$.
Define $m=\max_{t\in[a,b]}u(t)$. Using Lemma \ref{green1} and the facts
that $0\le u(t)\le m$, $u(t)\not\equiv m$ on $[a,b]$, and $q(t)\le q_{+}(t)$,
we have
$$
m<m\max_{t\in[a,b]}\int_{a}^{b}G(t,s)q_{+}(s)ds\le m\int_{a}^{b}G(s,s)q_{+}(s)ds.
$$
Canceling $m$ from both sides and using Lemma \ref{green1} again, we see that
$$
1<\dfrac{1}{\Gamma(\alpha)}\left(\frac{(b-a)(\alpha-1)}{2\alpha-2-\beta}\right)^{\alpha-1}\left(\frac{\alpha-1-\beta}{2\alpha-2-\beta}\right)^{\alpha-1-\beta}\int_{a}^{b}q_{+}(t)dt,
$$
which gives the desired result.
\end{proof}

\begin{remark} Notice when $\beta=0$, we obtain the improved form
\eqref{1.6} which was
the result presented by Ferreira in \cite{F2} and later was noted by Dhar and Kong in
\cite{DK5}.
Also, by setting $\alpha=2$ and $\beta=0$, we obtain the classical Lyapunov inequality.
\end{remark}

\section{Application to Boundary Value Problems}

In the last section, we apply the obtained results in Section 3 to study the nonexistence, uniqueness, 
and existence-uniqueness of solutions of related fractional-order BVPs.
First, we provide a sufficient condition for the nonexistence of a nontrivial solution of the BVP \eqref{1}, \eqref{2}.

\begin{theorem}\label{t6.1}
Assume
\begin{equation}\label{5.3}
\int_{a}^{b}q_{+}(t)dt\le \Gamma(\alpha)\left(\frac{2\alpha-2-\beta}{(b-a)(\alpha-1)}\right)^{\alpha-1}\left(\frac{2\alpha-2-\beta}{\alpha-1-\beta}\right)^{\alpha-1-\beta}.
\end{equation}
Then \eqref{1}, \eqref{2} has no nontrivial solutions.

\end{theorem}
\begin{proof}
Assume the contrary, i.e., BVP \eqref{1}, \eqref{2} has a nontrivial solution
$u$. Then by Theorem \ref{main1}, \eqref{lti1} holds. This contradicts assumption
\eqref{5.3}.
\end{proof}

Now we consider a nonlinear fractional BVP consisting of the equation
\begin{equation}
\Da u+f(t,u)=0,\label{5.5}
\end{equation}
together with the BCs \eqref{2},
where $\alpha\in(1,2]$, $\beta\in[0,\alpha-1]$.
Here we present a criterion for the existence of a unique solution for
BVP \eqref{5.5}, \eqref{2}.

\begin{theorem}\label{t5.1}
Assume $f:[a,b]\times\mathbb{R}\to\mathbb{R}$ is continuous
and satisfies a uniform Lipschitz condition with respect to the second
variable on $[a,b]\times\mathbb{R}$ with Lipschitz constant $K$;
that is
\begin{equation}\label{5.7}
|f(t,u_{1})-f(t,u_{2})|\le K|u_{1}-u_{2}|,
\end{equation}
for all $(t,u_{1}),(t,u_{2})\in[a,b]\times\mathbb{R}$. If
\begin{equation}\label{5.8}
b-a<\left[\frac{(\alpha-\beta)^\alpha\Gamma(\alpha+1)}{K(\alpha-1)^{\alpha-1}}\right]^{\frac{1}{\alpha}},
\end{equation}
then BVP \eqref{5.5}, \eqref{2} has a unique solution on $[a,b]$.
\end{theorem}

\begin{proof}
Let $\mathcal{B}$ be the Banach space of continuous functions defined
on $[a,b]$ with the norm
$$
||u||=\max_{t\in[a.b]}|u(t)|.
$$
Now $u(t)$ is a solution of BVP \eqref{5.5} if and only if $u(t)$
satisfies the integral equation
$$
u(t)=\int_a^bG(t,s)f(s,u(s))ds.
$$
Define the operator $T:\mathcal{B}\to\mathcal{B}$ by
$$
Tu(t)=\int_a^bG(t,s)f(s,u(s))ds.
$$
Then $T$ is completely continuous. We claim that $T$ has a unique fixed point in
$\mathcal{B}$. In fact, for any $u_1,u_2\in\mathcal{B}$, we have
$$
|Tu_1(t)-Tu_2(t)|\le\int_a^b|G(t,s)||f(s,u_1(s)-f(s,u_2(s)))|ds.
$$
Since $G(t,s)\ge 0$ on $[a,b]\times[a,b]$ and $f$ satisfies \eqref{5.7}, we have
\begin{eqnarray}\nonumber
|Tu_1(t)-Tu_2(t)|&\le& K\int_a^bG(t,s)|u_1(s)-u_2(s)|ds\\
&\le&K||u_1-u_2||\int_a^bG(t,s)ds.\label{5.9}
\end{eqnarray}
From Lemma \ref{l3.2}, it follows that
$$
|Tu_1(t)-Tu_2(t)|\le K\frac{(\alpha-1)^{\alpha-1}}{(\alpha-\beta)^\alpha \Gamma(\alpha+1)}
(b-a)^{\alpha}||u_1-u_2||<||u_1-u_2||,
$$
where we have used \eqref{5.8}. Hence $T$ is a contraction mapping on
$\mathcal{B}$. By the contraction mapping theorem, we obtain the desired
result.

\end{proof}

\begin{remark}
It is easy to see that the results in Theorem \ref{t5.1} can be extended to
a nonlinear fractional BVP consisting of the equation \eqref{5.5} and
the following nonhomogeneous BC:
$$
u(a)=0,\quad \Db u(b)=k,
$$
where $k\in\mathbb{R}$. We leave the details to the interested reader.
\end{remark}

\section{Conclusion}

In this paper, we obtained a Lyapunov-type inequality for a fractional differential equation with a fractional boundary condition. The inequality obtained is an improvement and a generalization of inequalities that have been obtained in the past. The inequality was applied to show the existence and nonexistence of solutions to a nonlinear fractional boundary value problem.

\end{document}